\documentclass[11pt,a4paper,oneside]{amsart}
\usepackage[scale=0.77]{geometry}
\usepackage[colorlinks=true,citecolor=blue]{hyperref}
\newtheorem{theorem}{Theorem}[section]
\newtheorem{lemma}[theorem]{Lemma}

\newtheorem{claim}[theorem]{Claim}

\theoremstyle{definition}
\newtheorem{remark}[theorem]{Remark}
\newtheorem{define}[theorem]{Definition}
\newtheorem{problem}[theorem]{Problem}

\usepackage{txfonts}
\usepackage[T1]{fontenc}

\usepackage{verbatim}
\usepackage{parskip}
\usepackage{tikz}
\DeclareMathOperator{\midp}{mid}
\numberwithin{equation}{section}
\newcommand{\df}[1]{\emph{\color{black!70!blue}#1}}

\newcommand{\R}{\mathbb R}
\newcommand{\Ss}{\mathbb S}

\newcommand{\cal}{\mathcal}
\usepackage{enumerate}

\DeclareMathOperator{\conv}{conv}
\DeclareMathOperator{\rk}{rk}

\DeclareMathOperator{\bd}{bd}
\DeclareMathOperator{\join}{GJ}

\usetikzlibrary{calc}
\tikzset{
    add/.style args={#1 and #2}{
        to path={
 ($(\tikztostart)!-#1!(\tikztotarget)$)--($(\tikztotarget)!-#2!(\tikztostart)$)
  \tikztonodes},add/.default={.2 and .2}}
}  

\tikzset{
  mark coordinate/.style={inner sep=0pt,outer sep=0pt,minimum size=2pt,
    fill=black,circle}
}

\begin{document} 

\title[The Colored Hadwiger Transversal Theorem in $\R^d$]{The Colored
  Hadwiger Transversal Theorem in $\R^d$}  

\author{Andreas F. Holmsen$^\ast$}
\thanks{$^\ast$Supported by the Basic Science Research Program through the National Research Foundation of Korea funded by the Ministry of Education, Science and Technology (NRF-2010-0021048).} 
\address{A.~F.~Holmsen\\ Department of Mathematical Sciences\\ KAIST\\  Daejeon\\ South Korea} 
\email{andreash@kaist.edu}

\author{Edgardo Rold\'an-Pensado$^\dagger$}
\thanks{$^\dagger$Research was conducted while visiting KAIST in Daejeon, South Korea}
\address{E. Rold\'an-Pensado\\ Instituto de Matem\'aticas\\ UNAM\\ M\'exico, D.F.\\ Mexico}
\email{e.roldan@math.ucl.ac.uk}


\begin{abstract}
Hadwiger's transversal theorem gives necessary and sufficient conditions for a family of convex sets in the plane to have a line transversal. A higher dimensional version was obtained by Goodman, Pollack and Wenger, and recently a colorful version appeared due to Arocha, Bracho and Montejano. We show that it is possible to combine both results to obtain a colored version of Hadwiger's theorem in higher dimensions. The proofs differ from the previous ones and use a variant of the Borsuk-Ulam theorem. To be precise, we prove the following. Let $F$ be a family of convex sets in $\R^d$ in bijection with a family $P$ of points in $\R^{d-1}$. Assume that there is a coloring of $F$ with sufficiently many colors such that any colorful Radon partition of points in $P$ corresponds to a colorful Radon partition of sets in $F$. Then some monochromatic subfamily of $F$ has a hyperplane transversal. 
\end{abstract} 

\maketitle 

\section{Introduction} 

\subsection{Background}
Many classical theorems of convexity, such as the theorems of Carath\'eodory \cite{Car1907}, Helly \cite{Hel1921}, Kirchberger \cite{Kir1903}, and Tverberg \cite{Tve1966}, admit remarkable \emph{colorful} versions. The first of these, the colorful Carath\'eodory theorem, was discovered by B\'ar\'any \cite{Bar1982} and its dual version, the colorful Helly theorem, was independently discovered by Lov\'asz (see section 3 of \cite{Bar1982}). Apart from their inherent charm, these results also have deep applications which appear to be unaccessible by the classical versions of the theorems (see chapters 8-10 in \cite{Mat2002}). Recently a colorful version of Hadwiger's theorem \cite{Had1957} on common line transversals to families of convex sets in the plane was discovered by Arocha, Bracho, and Montejano \cite{ABM2008}. Just as Hadwiger's theorem has a generalization to hyperplane transversals in any dimension \cite{GP1988,PW1990}, they conjectured that there should also exist a colorful version in higher dimensions. In this note we make the first steps in establishing their conjecture. 

\subsection{Definitions} 
Recall Radon's theorem which states that any set of $k+2$ points in $\R^k$ can be partitioned into two parts whose convex hulls intersect (see \cite{Eck1993}). Moreover, this partition is unique if and only if any $k+1$ of the points are affinely independent. In general, a Radon partition of a set of points in $\R^k$ is a pair of disjoint subsets whose convex hulls intersect. The combinatorial data which records all the Radon partitions of a set of points in $\R^k$ is an invariant of the point set known as its \emph{order-type} (see chapter 9.3 in \cite{Mat2002} and section 2 in \cite{GP1988}). Radon partitions extend to families of sets in $\R^d$ in a straightforward way. Let $F$ be a family of sets in $\R^d$. A \df{Radon partition} of $F$ is a pair of subfamilies $(G_1, G_2)$ of $F$ such that $G_1\cap G_2 = \emptyset$ and $\conv G_1 \cap \conv G_2 \neq \emptyset$, where $\conv G_i$ denotes the convex hull of the union of the members of $G_i$.  

Let $F$ be a family of compact connected sets in $\R^d$. An affine hyperplane which meets every member of $F$ is called a \df{hyperplane transversal}. The relationship between hyperplane transversals and Radon partitions comes from the observation that if $F$ has a hyperplane transversal, then the set of all Radon partitions of $F$ ``spans'' the set of all Radon partitions of a point set $P$ in $\R^k$ for some $k < d$. To see this, simply choose, from each member of $F$, a point contained in the hyperplane transversal. (Naturally, $F$ may have other Radon partitions as well.) The idea behind the ``Hadwiger-type'' theorems is that this necessary condition is also sufficient.

\begin{define}[$k$-ordering]
  Let $F$ be a set or a family of sets. A \df{$k$-ordering} of $F$ is a bijection $\varphi : F \to P$ where $P$ is set of points which affinely span $\R^k$. 
\end{define} 

\begin{define}[Consistent $k$-ordering]
  Let $F$ be a family of sets in $\R^d$. A \df{consistent $k$-ordering} of $F$ is a $k$-ordering which respects the Radon partitions of $\varphi(F)$. That is, 
  \[\conv \varphi(F_1) \cap \conv \varphi(F_2) \neq \emptyset \implies \conv F_1 \cap \conv F_2 \neq \emptyset\]
  for any pair of subfamilies $F_1$ and $F_2$ of $F$.
\end{define}

The Pollack-Wenger theorem can now be stated as follows. 

\begin{theorem}[Pollack-Wenger \cite{PW1990}]
  A family of compact connected sets in $\R^d$ has a hyperplane transversal if and only if $F$ has a consistent $k$-ordering for some $0\leq k \leq d-1$.
\end{theorem} 

The history leading up to the Pollack-Wenger theorem starts with the observation that the case $k=0$ follows from Helly's theorem in $\R^1$. Next, Hadwiger \cite{Had1957} proved the case $d=2$ and $k=1$ under the additional assumption that the members of $F$ are \emph{pairwise disjoint}. More than two decades later, Katchalski extended Hadwiger's theorem to arbitrary dimension, still using the condition of pairwise disjointness. In 1988 Goodman and Pollack \cite{GP1988} proved the case for $k=d-1$ under a condition of \emph{separatedness} generalizing the disjointness condition. It was not until 1990 that Wenger \cite{Wen1990} removed the condition of pairwise disjointness in the case $d=2$ and $k=1$, which immediately implies Katchalski's result as well. Wenger's discovery showed that the condition of disjointness (and separatedness) was in fact a bit misleading, and the Pollack-Wenger theorem served as a common generalization of the various Hadwiger-type results. 

\subsection{The colorful version}
Given a family of sets $F$, an \df{$r$-coloring} of $F$ is a partition of $F$ into $r$ non-empty parts $F = F_1 \cup F_2\cup\dots\cup F_r$. Each $F_i$ is called a \df{monochromatic subfamily} of $F$. A \df{colorful subfamily} of $F$ is a subfamily $G\subset F$ such that $\lvert G\cap F_i\rvert \leq 1$ for all $1\leq i \leq r$. A \df{colorful Radon partition} is a Radon partition $(G_1,G_2)$ such that $G_1\cup G_2$ is colorful.

\begin{define}[Rainbow consistent $k$-ordering]
	Let $F$ be an $r$-colored family of sets in $\R^d$. A \df{rainbow consistent $k$-ordering} of $F$ is a $k$-ordering which respects the colorful Radon partitions of $\varphi(F)$. That is,
	\[
		\conv \varphi(F_1) \cap \conv \varphi(F_2) \neq \emptyset \implies \conv F_1 \cap \conv F_2 \neq \emptyset
	\]
	for any pair of subfamilies $F_1$ and $F_2$ of $F$ where $F_1\cup F_2$ is a colorful subfamily. 
\end{define} 

Arocha, Bracho, and Montejano \cite{ABM2008} discovered the first colorful version of Hadwiger's theorem, or rather a colorful version of Wenger's theorem.

\begin{theorem}
	Let $F$ be a 3-colored family of compact connected sets in $\R^d$. If $F$ has a rainbow consistent 1-ordering, then some monochromatic subfamily of $F$ has a hyperplane transversal. 
\end{theorem} 

As pointed out in \cite{ABM2008}, it is not hard to formulate the colorful version of the Pollack-Wenger theorem (which they conjectured is true). This leads us to the general ``colorful Hadwiger problem''. 

\begin{problem}\label{prob:main}
	For integers $d$ and $k$, $d > k \geq 0$, determine the smallest integer $r = r(d,k)$ such that if $F$ is an $r$-colored family of compact connected sets in $\R^d$ with a rainbow consistent $k$-ordering, then some monochromatic subfamily of $F$ has a hyperplane transversal.
\end{problem} 

Note that existence of such and $r$ implies the Pollack-Wenger theorem. To see this, assume that $F$ is a family of compact connected sets in $\R^d$ and take $r$ monochromatic copies $F_1,\dots,F_r$ of $F$, each of a different color. Clearly, if $F$ has a consistent $k$-ordering then $\bigcup F_i$ has a rainbow consistent $k$-ordering and therefore some $F_i$ has a hyperplane transversal. Since $F_i$ is a copy of $F$, then $F$ also has a hyperplane transversal.

Before getting to our results, let us point out several known bounds for $r(d,k)$. 

\begin{itemize} 
	\item $r(d,k) \geq k+2$. If $F$ is colored by less than $k+2$ colors, then a bijection $\varphi: F \to P$ where $P$ is any set in general position in $\R^k$ is a rainbow consistent $k$-ordering since the set of colorful Radon partitions of $\varphi(F)$ will be empty.
	\item $r(d,k) \geq r(d+1,k)$. If $F$ is an $r$-colored family of compact connected sets in $\R^{d+1}$, let $\pi(F)$ denote the family obtained by projection to $\R^d$. Any rainbow consistent $k$-ordering of $F$ is also a rainbow consistent $k$-ordering of $\pi(F)$, and the preimage of a hyperplane in $\R^d$ is a hyperplane in $\R^{d+1}$.
	\item $r(1,0) = 2$. If $F$ is a 2-colored family in $\R^1$, then a rainbow consistent $0$-ordering simply means that any two members of $F$ of distinct colors have a point in common. The colorful Helly theorem implies that there is a monochromatic subfamily whose members have a point in common. 
	\item $r(2,1) = 3$. This is the colorful Hadwiger theorem of Arocha, Bracho, and Montejano \cite{ABM2008}. 
\end{itemize} 

In this note we present an approach which reduces Problem \ref{prob:main} to showing that a certain type of subsets are contractible. In the uncolored case these subsets are convex, so we obtain a new proof of the Pollack-Wenger theorem. In the colored case, these subsets correspond to what were called \emph{geometric joins} in \cite{BHK2013}, where their topology was studied. Based on results and ideas from \cite{BHK2013} we obtain the following.

\begin{theorem}\label{thm:summary}
	For the function $r(d,k)$ the following bounds hold. 
	\begin{itemize} 
		\item $r(k+2,k) \leq \binom{k+2}{2} + 1$.
		\item $r(4,2) = 4$.
		\item $r(k+1,k) \leq 2(k+1)^2+3$.
	\end{itemize} 
\end{theorem} 

Our proof method also gives a new (and simpler) proof of the Arocha-Bracho-Montejano theorem. 

\section{Proof of Theorem \ref{thm:summary}}

It is sufficient to prove Theorem \ref{thm:summary} for finite families. The general case follows from a standard compactness argument. Moreover, we may assume that the members of $F$ are convex polytopes. To see this, note that a hyperplane meets a compact connected set if and only if it meets the convex hull of the set. Thus we may assume the members of $F$ are convex. Next, we can approximate each convex set $K \in F$ by an inscribed convex polytope $K' \subset K$, forming a new family $F'$ such that the corresponding $k$-ordering, $\varphi'\colon F' \to P$, is rainbow consistent. This follows from the compactness of the members of $F$, and since the polytopes are inscribed, any hyperplane that intersects $K'$ also intersects $K$. So from here on, we assume $F$ is a finite family of polytopes in $\R^d$. 

Let $V$ be the set of vertices of the polytopes in $F$ and $\midp(V)$ the set of midpoints between pairs of points of $V$. For each pair of distinct points $u$ and $v$ in $V \cup \midp(V)$ consider the orthogonal complement $(u-v)^\bot$ which is a hyperplane through the origin in $\R^d$. The set of all such orthogonal complements decomposes $\Ss^{d-1}$ into a regular antipodal cell complex of dimension $d-1$ which is denoted by $\cal C$. The cells of $\cal C$ are \emph{open} and the boundary of $\sigma\in \cal C$, denoted by $\bd(\sigma)$, is a finite union of cells of $\cal C$. Note that $\cal C$ is homeomorphic to a polytopal complex 

The main step consists in assigning to each cell, $\sigma \in \cal C$, a subset $S(\sigma)\subset \R^m$ with the following properties. (The dimension $m$ depends on the values of $d$ and $k$ and will be determined later.)

\begin{description} 
	\item[Antipodality] $S(\sigma) = -S(-\sigma)$ for every $\sigma\in \cal C$.
	\item[Monotonicity] $S(\tau) \subset S(\sigma)$ for every $\tau, \sigma \in \cal C$ with $\tau \subset \bd(\sigma)$.
	\item[Contractibility] $S(\sigma)$ is contractible for every $\sigma \in \cal C$.
\end{description} 

\begin{lemma}\label{lem:hero}
	Suppose the sets $S(\sigma)\subset \R^m$ satisfy antipodality, monotonicity, and contractibility. If $m < d$, then one of the set $S(\sigma)$ contains the origin.
\end{lemma} 

\begin{proof}
	We construct a continuous, antipodal map, $f : \Ss^{d-1} \to \R^m$, by building it up inductively on the skeletons of $\cal C$. First define the map $f$ on the $0$-skeleton by choosing, for every $v\in \Ss^{d-1}$ corresponding to a $0$-cell of $\cal C$, an arbitrary point $y\in S(v)$, and set $f(v) = y$ and $f(-v) = -y$. Now suppose $f$ has been defined on the $k$-skeleton of $\cal C$ in such a way that for every cell $\tau$, its image $f(\tau)$ is contained in $S(\tau)$. Let $\sigma \in \cal C$ be a $(k+1)$-cell. The function $f$ has already been defined on $\bd(\sigma)$, which is homeomorphic to the $k$-sphere, and the monotonicity property implies that the image $f(\bd(\sigma)) = \bigcup_{\tau \subset \bd(\sigma)} f(\tau)$ is contained in $S(\sigma)$. Since $S(\sigma)$ is contractible, it is necessarily $k$-connected, and therefore $f$ can be extended continuously on all of $\sigma$ such that its image lies in $S(\sigma)$. Once $f$ has been defined on $\sigma$, extend antipodally on $-\sigma$. We can extend $f$ to the entire $(k+1)$-skeleton by repeating the procedure for every $(k+1)$-cell. We therefore have a continuous, antipodal map, $f : \Ss^{d-1} \to \R^m$. If $m < d$, then $f$ has a zero by the Borsuk-Ulam theorem (see e.g. \cite{Mat2003}) implying that there is a cell $\sigma \in \cal C$ such that $S(\sigma)$ contains the origin.
\end{proof} 

\begin{remark}
	The contractibility condition in Lemma \ref{lem:hero} can be replaced by the following weaker condition: If $\sigma \in \cal C$ is a $k$-cell then $S(\sigma)$ is $(k-1)$-connected.
\end{remark}

\subsection{Construction of \texorpdfstring{$S(\sigma)$}{S(s)}}
Let $F$ be a finite $r$-colored family of convex polytopes in $\R^d$. Suppose that no monochromatic subfamily of $F$ has a hyperplane transversal. 

\subsubsection{Separated subfamilies}
Identify the point $\textbf x = (x,t) \in \Ss^{d-1}\times \R$ with the hyperplane $H(\textbf x) = \{v\in \R^d : v\cdot x = t\}$, which should be thought of as an \emph{oriented} hyperplane, in the sense that $H(\textbf x)$ bounds a \emph{negative} and a \emph{positive} half-space, where the direction $x\in \Ss^{d-1}$ points to the positive side. Thus, the hyperplanes corresponding to $\textbf x$ and $-\textbf x$ determine the same point set but have reverse orientations. The space of all oriented hyperplanes in $\R^d$ is parametrized by $\Ss^{d-1}\times \R$ and comes equipped with the natural topology. In other words, we consider the space of all oriented hyperplanes as a ``$\mathbb{Z}_2$-space''.

For every oriented hyperplane $H \subset \R^d$ there is a corresponding ordered pair of \df{separated subfamilies} $(F_1,F_2)$, where $F_1 \subset F$ consists of the members contained in \emph{open negative} side of $H$, and $F_2 \subset F$ the members in \emph{open positive} side. The map $\textbf x \mapsto (F_1,F_2)$ is a ``$\mathbb Z_2$-map'' in the sense that if $\textbf x$ is mapped to $(F_1, F_2)$, then $-\textbf x$ is mapped to $(F_2,F_1)$. We write $(F_1,F_2) \subset (F_1',F_2')$ if $F_1\subset F_1'$ and $F_2\subset F_2'$, and $(F_1,F_2) = (F_1',F_2')$ in the case of equality. 

\begin{claim} \label{claim:neib}
	Every $\textbf x \in \Ss^{d-1}\times \R$ is contained in an open neighborhood $N(\textbf x)$ such that if $\textbf x$ corresponds to the separated subfamilies $(F_1,F_2)$ and $\textbf x'$ corresponds to the separated subfamilies $(F_1',F_2')$, then $(F_1,F_2)\subset (F_1',F_2')$ for any $\textbf x'\in N(\textbf x)$.
\end{claim} 

\begin{proof}
	Each member of $F_1 \cup F_2$ has some positive distance to the hyperplane $H(\textbf x)$, and since $\lvert F_1 \cup F_2\rvert$ is finite, a minimum distance is achieved. The distance to each member varies continuously with $\textbf x$, and consequently the distance from each member of $F_1\cup F_2$ to $H(\textbf x')$ remains positive for any $\textbf x'$ sufficiently close to $\textbf x$.
\end{proof} 

\subsubsection{Central hyperplanes}
We now define a specific hyperplane for every direction $x\in \Ss^{d-1}$ as follows. Consider an oriented hyperplane orthogonal to $x$ such that every member of $F$ is contained on its positive side. Start translating the hyperplane in the direction $x$ until the first time its closed negative side contains members of $F$ of \emph{at least $\lceil \frac{r}{2} \rceil$ distinct colors}, and denote this hyperplane $H_1$. Similarly, starting with a hyperplane which contains every member of $F$ on its negative side, we translate it in the direction $-x$ until the first time its closed positive side contains members of $F$ of \emph{at least $\lceil \frac{r}{2} \rceil$ distinct colors}, and denote this hyperplane $H_2$.

The assumption that no monochromatic subfamily of $F$ has a hyperplane transversal implies that $H_2$ is contained in the open positive side of $H_1$. If this were not the case, the ordered pair of separated subfamilies $(F_1,F_2)$ associated with $H_1$ would both contain strictly less than $\lceil \frac{r}{2} \rceil$ colors, implying that $F_1 \cup F_2$ contains strictly less than $r$ colors, hence $H_1$ is a hyperplane transversal to some monochromatic subfamily. For the direction $x \in \Ss^{d-1}$, let the \df{central hyperplane in the direction $x$} be the oriented hyperplane which is orthogonal to the direction $x$ and lies halfway between $H_1$ and $H_2$.

\begin{claim} \label{claim:separations}
	For $x \in \Ss^{d-1}$, let $H_x$ be the central hyperplane in the direction $x$ and $(F_x^-,F_x^+)$ the associated separated subfamilies. The following hold.
	\begin{enumerate}
		\item The map $x\mapsto H_x$ is a continuous, antipodal map from $\Ss^{d-1}$ to $\Ss^{d-1}\times \R$.
		\item $H_x$ passes through a midpoint determined by the vertices of the members of $F$.
		\item $F_x^-\cup F_x^+$ contains members of every color.
		\item $F_x^- $ and $F_x^+$ each contain members of at least $\lceil \frac{r}{2} \rceil$ distinct colors.
		\item $F_x^- = F_{-x}^+$ and $F_x^+ = F_{-x}^-$. 
	\end{enumerate} 
\end{claim} 

\begin{proof}
	For part \textit{(1)}, continuity follows from the continuity of the distance function, while the antipodality follows from the symmetry in the definition of the hyperplanes $H_1$ and $H_2$ (in the definition of the central hyperplane). For part \textit{(2)} we observe that the hyperplanes $H_1$ and $H_2$ are supporting tangents of members of $F$, and therefore must each contain at least one vertex of a member of $F$. Since $H_x$ lies halfway between $H_1$ and $H_2$, it must pass through the midpoint of these vertices. Part \textit{(3)} is just the assumption that no monochromatic subfamily has a hyperplane transversal. Part \textit{(4)} follows from the observation that $H_1$ lies in the open negative side of $H_x$ while $H_2$ lies in the open positive side. Part \textit{(5)} is a consequence of the antipodality of the map $x\mapsto H_x$. 
\end{proof} 

\begin{claim}\label{claim:cells}
	Let $\sigma$ be a cell of $\cal C$. Then $(F_x^-,F_x^+) = (F_y^-,F_y^+)$ for all $x$ and $y$ in $\sigma$.
\end{claim} 

\begin{proof}
	Notice that a change in $F_x^-$ (or $F_x^+$) occurs only if some member of $F$ becomes tangent to $H_x$ (as $x$ varies continuously). When this happens, $H_x$ passes through a vertex $v\in V$, and by Claim \ref{claim:separations} \textit{(2)}, $H_x$ also passes through a midpoint $m\in \midp(V)$. This means that when $H_x$ becomes tangent to $v$, the vector $x$ will enter the orthogonal complement $(v-m)^\bot$, thus leaving the open cell $\sigma$.
\end{proof} 

In view of Claim \ref{claim:cells}, the ordered pairs of separated subfamilies may be associated with the cells of $\cal C$ (rather than the points of $\Ss^{d-1}$). For a cell $\sigma \in \cal C$ we write $(F_\sigma^-,F_\sigma^+)$.

\begin{claim}\label{claim:monot}
	Let $\tau$ and $\sigma$ be cells of $\cal C$. If $\tau \subset \bd(\sigma)$, then $(F_\tau^-,F_\tau^+) \subset (F_\sigma^-,F_\sigma^+)$.
\end{claim} 

\begin{proof}
	For any point $x\in \tau$ there is an open neighborhood $N(x)\subset \Ss^{d-1}$ such that for all $y\in N(x)$ we have $y\in \sigma$ for some $\sigma \in \cal C$ with $\tau\subset \bd(\sigma)$. The statement now follows by taking the intersection with the open neighborhood from Claim \ref{claim:neib}.
\end{proof}

\subsubsection{The colorful Radon partitions}
	Let $P$ be a set of points which affinely spans $\R^k$ and $\varphi : F \to P$ a $k$-ordering. Let $P^+\in \R^{k+1}$ be the set of points obtained by adding a coordinate with value $1$ to the end of each point in $P$, and let $P^- = -P^+$. For the cell $\sigma \in \cal C$, with associated separated subfamilies $(F_\sigma^-,F_\sigma^+)$, define sub-configurations $Q_\sigma^- \subset P^-$ and $Q_\sigma^+ \subset \cup P^+$ as
	\[
		Q_\sigma^- := \left\{ -\binom{p}{1} \, : \, \varphi(p) \in F_\sigma^- \right\} \quad \text{ and } \quad Q_\sigma^+ := \left\{ \binom{p}{1} \, : \, \varphi(p) \in F_\sigma^+ \right\} 
	\]

\begin{lemma}\label{lem:mono-anti}
	Let $Q_\sigma = Q_\sigma^- \cup Q_\sigma^+$. The following hold. 
	\begin{itemize} 
		\item $Q_\sigma = -Q_{-\sigma}$ for every $\sigma \in \cal C$. (Antipodality) 
		\item $Q_\tau \subset Q_\sigma$ for every $\tau, \sigma \in \cal C$ with $\tau\subset \bd(\sigma)$. (Monotonicity)
	\end{itemize} 
\end{lemma} 

\begin{proof}
	Antipodality follows from Claim \ref{claim:separations} \textit{(5)}, while monotonicity follows from Claim \ref{claim:monot}. 
\end{proof} 

Since each point in $Q_\sigma$ corresponds to a unique member of $F$, Claim \ref{claim:separations} \textit{(3)} implies that there is a natural $r$-coloring of $Q_\sigma$. We may therefore speak of the \emph{colorful subsets} of $Q_\sigma$. A crucial observation is the following.

\begin{lemma} \label{lem:norigin}
	If $Q$ is a colorful subset of $Q_\sigma$ and the convex hull of $Q$ contains the origin, then $\varphi: F \to P$ is not rainbow consistent.
\end{lemma} 

\begin{proof}
	Let $Q_1 := -(Q \cap Q_\sigma^-)$ and $Q_2 := Q \cap Q_\sigma^+$. The fact that $0 \in \conv(Q)$ is equivalent to saying that $\conv (Q_1) \cap \conv(Q_2) \neq \emptyset$, hence if $0\in \conv(Q)$, then this corresponds to a colorful Radon partition $(P_1, P_2)$ of $P$ (by dropping the last coordinate). However, by definition of $Q_\sigma$, the subfamilies $\varphi^{-1}(P_1) = F_\sigma^-$ and $\varphi^{-1}(P_2) = F_\sigma^+$ are strictly separated by a central hyperplane. 
\end{proof} 

\subsection{The geometric join}
We are now ready to define the sets $S(\sigma)$. This will vary slightly depending on which case of Theorem \ref{thm:summary} we want to prove. 

\begin{define}[Geometric join]
	Let $A = A_1 \cup A_2 \cup\dots\cup A_r$ be an $r$-colored set of points in $\R^{k+1}$. The \df{geomteric join $\join(A)$} of $A$ is the set of all convex combinations of the form $t_1a_1 + t_2a_2 + \cdots + t_ra_r$ where $a_i\in A_i$.
\end{define} 

We need the following results from \cite{BHK2013}.

\begin{lemma}\label{lem:join}
	The geometric join of an $r$-colored point set in $\R^{k+1}$ is contractible in the following cases 
	\begin{itemize} 
		\item $k = 1$ and $r = 3$.
		\item $k = 2$ and $r = 4$.
		\item $k\geq 3$ and $r = \binom{k+2}{2} +1$.
	\end{itemize} 
\end{lemma} 

\subsubsection{The case $d = k+2$}
We show that $r(k+2,k) \leq \binom{k+2}{2}+1$ for $k\geq 3$, and $r(4,2)=4$. For every $\sigma \in \cal C$, the set $Q_\sigma$ is an $r$-colored point set in $\R^{k+1}$. Let $S(\sigma)=\join(Q_\sigma)$. By Lemma \ref{lem:mono-anti} the sets $S(\sigma)$ satisfy antipodality and monotonicity. Contractibility follows from Lemma \ref{lem:join} provided
\[
	r = \left\{
	\begin{array}{ll}
		4 & \mbox{if } k = 2 \\ \binom{k+2}{2}+1 & \mbox{if } k\geq 3 
	\end{array}\right.
\]
Lemma \ref{lem:hero} implies that the origin is contained in $S(\sigma)$ for some $\sigma\in \cal C$, and Lemma \ref{lem:norigin} implies that $\varphi\colon F \to P$ is not rainbow consistent. 

\subsubsection{The case $d = k+1$}
We show that $r(k+1,k)\leq 2(k+1)^2+3$. For every $\sigma \in \cal C$, the set $Q_\sigma$ is an $r$-colored point set in $\R^{k+1}$. Let $W$ be a hyperplane passing through the origin which strictly separates $P^-$ and $P^+$. Let $S(\sigma)=W\cap\join(Q_\sigma)$, note that $S(\sigma)$ is a subset of $\R^k$.

\begin{lemma}\label{lem:stars}
	If $r=2(k+1)^2+3$, then $S(\sigma)$ is star-shaped for every $\sigma \in \cal C$. 
\end{lemma} 

\begin{proof}
	We will apply Krasnoselskii's theorem (see e.g. Theorem 11.2 in \cite{Eck1993}), and since $S(\sigma) \subset \R^{k}$ is compact, it suffices to show that every $k+1$ boundary points of $S(\sigma)$ are visible from a common point of $S(\sigma)$. Let $x_0$, $x_1$, $\dots$, $x_{k}$ be boundary points of $S(\sigma)$, that is, $x_i \in W \cap \conv(X_i)$ where $X_i$ is a colorful subset of $Q_\sigma$ with $\lvert X_i\rvert\leq k+1$. Recall that the points of $Q_\sigma$ are in bijection with the members of $F_\sigma^- \cup F_\sigma^+$, so by Claim \ref{claim:separations} \textit{(4)} it follows that $Q_\sigma^-$ and $Q_\sigma^+$ each contain points of at least $(k+1)^2+2$ distinct colors. Since $\lvert X_0 \cup \dots \cup X_{k}\rvert \leq (k+1)^2$ there exists a point $p_1 \in Q_\sigma^-$ such that $X_i \cup \{p_1\}$ is colorful for all $0\leq i \leq k$. Similarly, since $\lvert X_1 \cup \cdots \cup X_d \cup \{p_1\}\rvert \leq (k+1)^2+1$ there exists a point $p_2\in Q_\sigma^+$ such that $X_i \cup \{p_1, p_2\}$ is colorful for all $0\leq i \leq k$. Since $p_1\in Q_\sigma^-$ and $p_2\in Q_\sigma^+$ are strictly separated by $W$, the segment $p_1p_2$ intersects $W$ in a unique point, $p$, and since $\conv(X_i \cup \{p_1, p_2\}) \cap W$ is a convex subset contained in $S(\sigma)$ it follows that all the $x_i$ are visible from $p$.
\end{proof} 

By Lemmas \ref{lem:mono-anti} and \ref{lem:stars} the sets $S(\sigma)$ satisfy antipodality, monotonicity, and contractibility. So, by Lemma \ref{lem:hero} there is a $\sigma \in \cal C$ such that $S(\sigma)$ contains the origin, and therefore some colorful subset of $Q_\sigma$ contains the origin in its convex hull. By Lemma \ref{lem:norigin}, $\varphi: F \to P$ is not rainbow consistent.

\subsubsection{The case $d=2$}
We give an alternate proof to the Arocha-Bracho-Montejano theorem which states that $r(2,1)=3$. For every $\sigma\in\cal C$, the set $Q_\sigma$ is a $3$-colored point set in $\R^2$. Let $W$ be a line through the origin that strictly separates $P^-$ and $P^+$. Let $S(\sigma)$ be the convex hull of $W\cap\join(Q_\sigma)$.

\begin{lemma}\label{lem:d=2}
	If $S(\sigma)$ contains the origin, then $\join(Q_\sigma)$ also contains the origin.
\end{lemma}

\begin{proof}
	Assume that $S(\sigma)$ contains the origin. There are point $x_1,x_2\in Q^+_\sigma$ and $x_3,x_4\in Q^-_\sigma$ such that $\{x_1,x_3\}$ and $\{x_2,x_4\}$ are both colorful subsets of $Q_\sigma$ and the segments $x_1x_3$ and $x_2x_4$ intersect $W$ on different sides of the origin. By Claim \ref{claim:separations} {\em (4)}, $Q_\sigma^+$ and $Q_\sigma^-$ each contain at least 2 distinct colors, and it follows that $\join(Q_\sigma^+)$ and $\join(Q_\sigma^-)$ are both connected. Hence there is a path connecting $x_1$ to $x_2$ contained in $\join(Q_\sigma^+)$, and a path connecting $x_3$ to $x_4$ in $\join(Q_\sigma^-)$. By combining these paths together with the segments $x_1x_3$ and $x_2x_4$, we obtain a closed cycle contained in $\join(Q_\sigma)$ which encloses the origin. Since $\join(Q_\sigma)$ is contractible by Lemma \ref{lem:join} with $k=1$, $\join(Q_\sigma)$ contains the origin.
\end{proof}

By Lemmas \ref{lem:mono-anti} and \ref{lem:d=2}, $S(\sigma)$ satisfies the hypothesis of Lemma \ref{lem:hero}. Therefore there is some $\sigma \in \cal C$ such that $S(\sigma)$ contains the origin and Lemma \ref{lem:norigin} implies that $\varphi: F \to P$ is not rainbow consistent.

\section{Final remarks}

\subsection{The case \texorpdfstring{$d=3$}{d=3}}

The case $d=k+1$ is the most interesting one and, by comparing with other colorful theorems, we expect $r(k+1,k)=k+2$. This is in fact true for $k=1$ as was shown by Arocha, Bracho and Montejano, however the value of $r(3,2)$ remains unknown and we were unable to adapt out methods to determine its value. It seems that this problem is strongly connected to the topology of geometric joins.

\subsection{Matroids}

Theorem \ref{thm:summary} can be generalized further by using a matroid instead of colors. Let $\cal M$ be a simple matroid of rank $r$ with rank function $\rk(\cdot)$, whose ground set is the family $F$. The colorful case occurs when $\cal M$ is a partition matroid. That is, when the elements of $F$ are partitioned into non-empty sets $F=F_1\cup\dots\cup F_r$ and a subset $G\subset F$ is independent if and only if $\lvert G\cap F_i\rvert\le 1$ for all $1\le i\le r$.

The notions of \df{independent Radon partition} and \df{independent consistent $k$-ordering} can be defined in terms of $\cal M$ in a natural way. The number $\overline r=\overline r(d,k)$ can be defined as the smallest integer such that if $\cal M$ is a rank $\overline r$ matroid on a family $F$ of compact connected sets in $\R^d$ with an independent consistent $k$-ordering, then there is a subfamily $G$ of $F$ such that $\rk(F\setminus G)<\overline r$ and $G$ has a hyperplane transversal.

\begin{theorem}
	For the function $\overline r(d,k)$ the following bounds hold. 
	\begin{itemize} 
		\item $\overline r(3,2) = 3$.
		\item $\overline r(k+1,k) \leq 2(k+1)^2+3$.
	\end{itemize} 
\end{theorem}

The proof of this theorem is almost identical to the proof of Theorem \ref{thm:summary}, so we omit it.

\end{document}